\documentclass[12pt]{article}
\usepackage{amsmath,amsfonts,amsthm,amssymb, mathtools}
\usepackage{mathrsfs, graphicx,color,latexsym, tikz, calc,cite,enumerate,indentfirst}
\usetikzlibrary{shadows}
\usetikzlibrary{patterns,arrows,decorations.pathreplacing}
\usepackage[colorlinks=true,
linkcolor=blue,citecolor=blue,
urlcolor=blue]{hyperref}

\newtheorem{theorem}{Theorem}[section]
\newtheorem{lemma}[theorem]{Lemma}
\newtheorem{problem}[theorem]{Problem}

\newtheorem{corollary}[theorem]{Corollary}
\newtheorem{remark}[theorem]{Remark}

\renewcommand\proofname{\it{Proof}}

\allowdisplaybreaks

\title{\bf An Extension of  P\'{o}lya's Enumeration Theorem}

\author{Xiongfeng Zhan, Xueyi Huang\footnote{Corresponding author.} \setcounter{footnote}{-1}\footnote{\emph{Email address:} zhanxfmath@163.com (X. Zhan), huangxy@ecust.edu.cn (X. Huang).}\\[2mm]
	\small School of Mathematics, East China University of Science and Technology, \\
	\small  Shanghai 200237, P. R. China
}

\date{}

\begin{document}
	\maketitle
	\begin{abstract}
		In combinatorics, P\'{o}lya's Enumeration Theorem is a powerful tool for solving a wide range of counting problems, including the enumeration of groups, graphs, and chemical compounds. In this paper, we present an extension of P\'{o}lya's Enumeration Theorem. As an application, we derive a formula that expresses  the  $n$-th elementary symmetric polynomial in $m$ indeterminates (where $n\leq m$) as a variant of the cycle index polynomial of the symmetric group $\mathrm{Sym}(n)$. This result resolves a problem posed by Amdeberhan in 2012.

		\par\vspace{2mm}
		
		\par\vspace{1mm}
		
	\end{abstract}

	\section{Introduction}\label{section::1}
	
	P\'{o}lya's Enumeration Theorem, also known as the P\'{o}lya-Redfield Theorem, is a fundamental tool in combinatorics that ultimately generalizes Burnside's Lemma on the number of orbits of a group action on a set.  The theorem was first introduced by Redfield in 1927 \cite{Red27}. A decade later, P\'{o}lya \cite{Pol37} independently rediscovered it in 1937  and demonstrated its utility in solving diverse counting problems, including the enumeration of groups, graphs, and chemical compounds. Over time, P\'{o}lya's Enumeration Theorem has become a cornerstone of combinatorial mathematics, inspiring widespread applications and further theoretical developments \cite{PR87}.

 Let $n$ and $m$ be two fixed  positive integers. Let $F$ be a field  of characteristic zero, and let  $F[t_1,t_2,\ldots,t_n]$ (resp. $F[w_1,w_2,\ldots,w_m]$) denote the ring of polynomials over $F$ in indeterminates $t_1,\ldots,t_n$ (resp. $w_1,\ldots,w_m$).
	
	Let $X$ be an $n$-element set, and let $G$ be a permutation group on $X$, \textit{i.e.}, a subgroup of the symmetric group $\mathrm{Sym}(X)$. For $\sigma\in G$, let $c_i(\sigma)$ denote the number of cycles of length $i$ in the cycle decomposition of $\sigma$ as an element of  $\mathrm{Sym}(X)$. The \textit{cycle index polynomial} of  $G$  is defined as 
	\[
		Z_{G}(\boldsymbol{t})=\frac{1}{|G|}\sum_{\sigma\in G}Z(\sigma,\boldsymbol{t})\in F[t_1,t_2,\ldots,t_n],
	\]
	where $\boldsymbol{t}=(t_1,t_2,\ldots,t_n)$ and 	\[Z(\sigma,\boldsymbol{t})=t_1^{c_1(\sigma)}t_2^{c_2(\sigma)}\cdots t_n^{c_n(\sigma)}.
	\]

	Let $Y=\{y_1,\ldots,y_m\}$ be an $m$-element set, and let $Y^X$ denote the set of all functions from $X$ to $Y$.  Define a weight function  $w:Y\longrightarrow F[w_1,w_2,\ldots,w_m]$ by assigning  $w(y_i)=w_i$ for  each $i\in [m]$. The \textit{weight of a function} $f\in Y^X$ is defined as 
	\begin{equation*}
		W(f)=\prod_{x\in X}w(f(x)).
	\end{equation*}
Observe that $W(f)$ is a monomial of degree $n$ in $F[w_1,w_2,\ldots,w_m]$. 

For any $\sigma\in G$ and $f\in Y^X$, let $f^{\sigma}$ denote the function defined by $f^{\sigma}(x)=f(x^{\sigma})$. This action of $\sigma$ induces a permutation $\sigma'$ on $Y^X$, given by $\sigma'(f)=f^{\sigma}$. Thus, $G$ acts as a permutation group on $Y^X$. Let $f^G=\{f^{\sigma}:\sigma\in G\}$ denote the orbit of $f$, and let  $G_f=\{\sigma\in G:f^{\sigma}=f\}$ be the stabilizer subgroup of $f$. Define an equivalence relation $\sim$ on $Y^X$ by $f\sim h$ if and only if $h\in f^G$. The set of equivalence classes under $\sim$ is denoted by  $Y^X/G$. It follows directly that  $W(f)=W(h)$ whenever $f\sim h$. Therefore, the \textit{weight of an equivalence class}  $\mathcal{O}\in Y^X/G$ is well-defined as  $W(\mathcal{O})=W(f)$ for any $f\in\mathcal{O}$. 

Define the set
	$$\mathcal{P}_m([n])=\left\{\boldsymbol{k}=(k_1,k_2,\ldots,k_m):k_i\in [n]\cup\{0\} ~\mbox{for}~i\in [m],\sum_{i=1}^mk_i=n\right\}.$$ 	For any $\boldsymbol{k}=(k_1,k_2,\ldots,k_m)\in \mathcal{P}_m([n])$, write $\boldsymbol{w}^{(\boldsymbol{k})}=w_1^{k_1}w_2^{k_2}\cdots w_m^{k_m}$, and let   $a_{\boldsymbol{k}}$ denote the number of equivalence classes $\mathcal{O}\in Y^X/G$ with $W(\mathcal{O})=\boldsymbol{w}^{(\boldsymbol{k})}$. Let 
    $\boldsymbol{w}=(w_1,w_2,\ldots,w_m)$ and define  \[\tilde{\boldsymbol{w}}=\left(\sum_{i=1}^m w_i,\sum_{i=1}^m w_i^2,\ldots,\sum_{i=1}^m w_i^n\right).\]

  The weighted version of P\'{o}lya's Enumeration Theorem expresses the generating function of the numbers $a_{\boldsymbol{k}}$ as the cycle index polynomial $Z_G(\tilde{\boldsymbol{w}})$ of $G$  (see \cite{St99}).
	
\begin{theorem}[P\'{o}lya's Enumeration Theorem]\label{thm::Polya}
		\[
		\sum_{\boldsymbol{k}\in \mathcal{P}_m([n]) }a_{\boldsymbol{k}}\boldsymbol{w}^{(\boldsymbol{k})}=Z_G(\tilde{\boldsymbol{w}})=\frac{1}{|G|}\sum_{\sigma\in G}Z(\sigma,\tilde{\boldsymbol{w}}).
		\]
\end{theorem}

In this paper, we extend P\'{o}lya's Enumeration Theorem as follows.
	\begin{theorem}\label{thm::main1}
		Let $\Delta:G\longrightarrow F$ be a  function. Then
		\[
			\sum_{\boldsymbol{k}\in \mathcal{P}_m([n])}\left(\sum_{f\in Y^X:W(f)=\boldsymbol{w}^{(\boldsymbol{k})}}\sum_{\sigma\in G_f}\Delta(\sigma)
\right)\boldsymbol{w}^{(\boldsymbol{k})}=\sum_{\sigma\in G}\Delta(\sigma)Z(\sigma,\tilde{\boldsymbol{w}}).
		\]
			\end{theorem}

\begin{remark}\rm			
			Note that if $\Delta(\sigma) = 1/|G|$ for all $\sigma \in G$, the  Orbit-Stabilizer Theorem implies that
			\[
			\begin{aligned}
			\sum_{f\in Y^X:W(f)=\boldsymbol{w}^{(\boldsymbol{k})}}\sum_{\sigma\in G_f}\frac{1}{|G|}&=\sum_{f\in Y^X:W(f)=\boldsymbol{w}^{(\boldsymbol{k})}}\frac{|G_f|}{|G|}=\sum_{f\in Y^X:W(f)=\boldsymbol{w}^{(\boldsymbol{k})}}\frac{1}{|f^G|}\\
			&=\sum_{\mathcal{O}\in  Y^X/G:W(\mathcal{O})=\boldsymbol{w}^{(\boldsymbol{k})}}1
			=a_{\boldsymbol{k}}.
			\end{aligned}
			\]
Thus, Theorem \ref{thm::Polya} is a special case of  Theorem \ref{thm::main1}.
\end{remark}
		
As an application of Theorem \ref{thm::main1} (or its equivalent formulation presented in Theorem \ref{thm::main3}), we derive the following formula that expresses the  $n$-th elementary symmetric polynomial in $m$ indeterminates ($n\leq m$) as a variant of the cycle index polynomial of the symmetric group $\mathrm{Sym}(n):=\mathrm{Sym}([n])$. 
	\begin{theorem}\label{thm::main2}
Let $n\leq m$, and let $e_n(\boldsymbol{w})$  be the $n$-th elementary symmetric polynomial in  indeterminates $w_1,w_2,\ldots,w_m$. Then 
		\[
		e_n(\boldsymbol{w})=\frac{1}{n!}\sum_{\sigma\in \mathrm{Sym}(n)}\mathrm{sgn}(\sigma)Z(\sigma,\tilde{\boldsymbol{w}}),
		\]
		where $\mathrm{sgn}(\sigma)$ is the sign of  $\sigma$.
	\end{theorem}
	
		In 2012, Amdeberhan \cite{AMD22} posed the following problem concerning a determinant expression involving traces of matrix powers.

\begin{problem}[{\cite[Problem 7.1]{AMD22}}]
Let $L$ be an $n\times n$ matrix over  $F$. Denote $t_i = \mathrm{tr}(L^i)$ as the trace of a matrix. Then, we have
\begin{equation}\label{eq::1}
\mathrm{det}(L)=\frac{1}{n!}\sum_{\sigma\in \mathrm{Sym}(n)}\mathrm{sgn}(\sigma)Z(\sigma,\boldsymbol{t}),
\end{equation}
where $\boldsymbol{t}=(t_1,t_2,\ldots,t_n)$.
It's probably not hard to employ symmetric functions in proving \eqref{eq::1}. However, the above determinant
behaves eerily similar to a cycle index formula. So, find an interpretation of \eqref{eq::1} as some appropriate
group action.
\end{problem} 
As a corollary of Theorem \ref{thm::main2}, we provide a  solution to Amdeberhan's problem.

	\begin{corollary}\label{cor::main}
		Let $L$ be an $n\times n$ matrix over  $F$ with $t_i=\mathrm{tr}(L^i)$. Then  
		\[
		\mathrm{det}(L)=\frac{1}{n!}\sum_{\sigma\in \mathrm{Sym}(n)}\mathrm{sgn}(\sigma)Z(\sigma,\boldsymbol{t}),
		\]
		where $\boldsymbol{t}=(t_1,t_2,\ldots,t_n)$.
	\end{corollary}

	\section{Proofs of main results}\label{section::2}
	
In this section, we first establish an equivalent formulation of Theorem \ref{thm::main1} and provide its complete proof. All notations and terminology from Section \ref{section::1} will remain consistent throughout.

	For every $f\in Y^X$, and for every $X'\subseteq X$, if, for every $x\in X'$, $f(x)=y$, then we write $f(X')=y$ rather than $f(X')=\{y\}$. 

Let $\boldsymbol{k} = (k_1, k_2, \ldots, k_m) \in \mathcal{P}_m([n])$.  
The set of all partitions of $X$ whose sets have cardinalities prescribed by the entries of $\boldsymbol{k}$ is denoted by $\boldsymbol{k}(X)$. Specifically,  
\[
    \boldsymbol{k}(X) = \left\{\boldsymbol{\alpha} = (A_1, \ldots, A_m) : A_i \subseteq X,\, |A_i| = k_i,\, A_i \cap A_j = \emptyset \text{ for all distinct } i, j \in [m]\right\}.
\]
For a partition $\boldsymbol{\alpha} = (A_1, \ldots, A_m) \in \boldsymbol{k}(X)$, let $\mathrm{Sym}(\boldsymbol{\alpha}) = \mathrm{Sym}(A_1) \times \cdots \times \mathrm{Sym}(A_m)$, and let $f_{\boldsymbol{\alpha}} \in Y^X$ be the function defined by setting $f(A_i) = y_i$ for all $i \in [m]$. Then $\mathrm{Sym}(\boldsymbol{\alpha})$ is a subgroup of $\mathrm{Sym}(X)$, and the mapping  
\[
    \Phi\colon \boldsymbol{k}(X) \longrightarrow \left\{f \in Y^X : W(f) = \boldsymbol{w}^{(\boldsymbol{k})}\right\},~ \boldsymbol{\alpha} \longmapsto f_{\boldsymbol{\alpha}}
\]
is a bijection. Consequently, for every $f \in Y^X$ with $W(f) = \boldsymbol{w}^{(\boldsymbol{k})}$, there exists a unique $\boldsymbol{\alpha} \in \boldsymbol{k}(X)$ such that $f = f_{\boldsymbol{\alpha}}$. In this case, we also have $G_f = G_{f_{\boldsymbol{\alpha}}} = \mathrm{Sym}(\boldsymbol{\alpha}) \cap G$. Thus, we obtain the following equivalent formulation of Theorem \ref{thm::main1}.

\begin{theorem}(Equivalent formulation of Theorem \ref{thm::main1})\label{thm::main3}
    Let $\Delta\colon G \to F$ be a function. Then
    \[
        \sum_{\boldsymbol{k} \in \mathcal{P}_m([n])} \left(\sum_{\boldsymbol{\alpha} \in \boldsymbol{k}(X)} \sum_{\sigma \in \mathrm{Sym}(\boldsymbol{\alpha}) \cap G} \Delta(\sigma)\right) \boldsymbol{w}^{(\boldsymbol{k})} = \sum_{\sigma \in G} \Delta(\sigma) Z(\sigma, \tilde{\boldsymbol{w}}).
    \]
\end{theorem}
	
In what follows, we shall prove Theorem \ref{thm::main3}, thereby providing  a proof of Theorem \ref{thm::main1}. 

For a given   $\boldsymbol{k}=(k_1,k_2,\ldots,k_m)\in \mathcal{P}_m([n])$, the set of all functions in $Y^X$ whose sequence of preimage cardinalities for the elements of $Y$ coincides with $\boldsymbol{k}$ is denoted by $\mathrm{Im}(\boldsymbol{k})$, that is,
	$$\mathrm{Im}(\boldsymbol{k})=\{f\in Y^X: |\{x\in X:f(x)=y_i\}|=k_i ~\mathrm{for}~ 1\leq i\leq m \}.$$ 	For any $\sigma\in G$, we define $$\mathcal{F}_\sigma=\{f\in Y^X: f(x^{\sigma})=f(x) ~\mathrm{for~ every}~ x\in X\}.$$ Note that $\mathcal{F}_\sigma$ is precisely the set of functions that are constant on each cycle of  $\sigma$. 
	
	To prove Theorem \ref{thm::main3}, we require the following lemma.
	\begin{lemma}\label{lem::key}
		For any $\sigma\in G$, 
		\begin{equation*}
			\sum_{f\in \mathcal{F}_{\sigma}}W(f)=Z(\sigma,\tilde{\boldsymbol{w}}).
		\end{equation*}
	\end{lemma}
	\begin{proof}
		Let $\Pi_{\sigma}:=\{X_1,X_2,\ldots,X_k\}$ be the partition of $X$ induced by the cycle decomposition of  $\sigma$. For each $l\in [n]$, $c_l(\sigma)$ is the number of parts $X_j\in\Pi_{\sigma}$ such that $|X_j|=l$.  Moreover, every function $f\in \mathcal{F}_{\sigma}$ is constant on each part $X_j$ of $\Pi_{\sigma}$.  That is, for each $j\in [k]$, there exists  $f_j\in Y$ such that $f(X_j)=f_j$.  Therefore, we can express $\mathcal{F}_{\sigma}$ as \[\mathcal{F}_{\sigma}=\{f\in Y^X: f_1,\ldots,f_k\in Y,f(X_1)=f_1,\ldots,f(X_k)=f_k\}.\] Now, we compute the sum of the weights of all functions in $\mathcal{F}_{\sigma}$:
		\begin{equation*}
			\begin{aligned}
				\sum_{f\in \mathcal{F}_{\sigma}}W(f)&=\sum_{\substack{f_1,\ldots,f_k\in Y\\ f(X_1)=f_1,\ldots,f(X_k)=f_k}}W(f)\\
				&=\sum_{\substack{f_1,\ldots,f_k\in Y\\ f(X_1)=f_1,\ldots,f(X_k)=f_k}}\prod_{j=1}^k\left(\prod_{x\in X_j} w(f(x))\right)\\
				&=\sum_{f_1,\ldots,f_k\in Y}\prod_{j=1}^k\left(w(f_j)\right)^{|X_j|}\\
				&=\prod_{j=1}^k\left(\sum_{f_j\in Y}(w(f_j))^{|X_j|}\right)\\
				&=\prod_{j=1}^k\left(\sum_{i=1}^mw_i^{|X_j|}\right)\\
				&=\prod_{l=1}^n\left(\sum_{i=1}^mw_i^l\right)^{c_l(\sigma)}\\
				&=Z(\sigma,\tilde{\boldsymbol{w}}),\\
			\end{aligned}
		\end{equation*}
		where the antepenultimate equality follows from the fact that $\{w(f_j):f_j\in Y\}=\{w_1,w_2,\ldots,w_m\}$.
		The result follows.
	\end{proof}

Now we are in a position to give the proof of Theorem \ref{thm::main3}.

\renewcommand\proofname{\it{Proof of Theorem \ref{thm::main3}}} 

	\begin{proof}
		First, observe that
		\begin{equation}\label{eq::2}
			\begin{aligned}
				\sum_{f\in \mathcal{F}_{\sigma}}W(f)&=\sum_{\boldsymbol{k}\in \mathcal{P}_m([n])}\sum_{f\in \mathcal{F}_{\sigma}\cap \mathrm{Im}(\boldsymbol{k})}W(f)\\
				&=\sum_{\boldsymbol{k}\in \mathcal{P}_m([n])}\sum_{f\in \mathcal{F}_{\sigma}\cap \mathrm{Im}(\boldsymbol{k})}\boldsymbol{w}^{(\boldsymbol{k})}\\
				&=\sum_{\boldsymbol{k}\in \mathcal{P}_m([n])} \left(\sum_{f\in \mathcal{F}_{\sigma}\cap \mathrm{Im}(\boldsymbol{k})}1\right) \boldsymbol{w}^{(\boldsymbol{k})} \\
				&=\sum_{\boldsymbol{k}\in \mathcal{P}_m([n])} \left(\sum_{\boldsymbol{\alpha}=(A_1,\ldots,A_m)\in \boldsymbol{k}(X)}\sum_{\substack{f\in \mathcal{F}_{\sigma}\\f(A_i)=y_i, i\in [m]}}1\right)  \boldsymbol{w}^{(\boldsymbol{k})}.\\
			\end{aligned}
		\end{equation}
		 Since $\mathcal{F}_{\sigma}$ consists of functions constant on the disjoint cycles of the decomposition of $\sigma$, we have that $f_{\boldsymbol{\alpha}}\in\mathcal{F}_{\sigma}$ if, and only if, $\sigma$ stabilizes the partition $\boldsymbol{\alpha}$, \textit{i.e.}, $\sigma\in \mathrm{Sym}(\boldsymbol{\alpha})$.
			Therefore, for any $\boldsymbol{\alpha}=(A_1,\ldots,A_m)\in \boldsymbol{k}(X)$, 
			\begin{equation}\label{eq::3}
				\sum_{\substack{f\in \mathcal{F}_{\sigma}\\f(A_i)=y_i, i\in [m]}}1=\sum_{\substack{f\in \mathcal{F}_{\sigma}\\f=f_{\boldsymbol{\alpha}}}}1=\chi_{\mathrm{Sym}(\boldsymbol{\alpha})}(\sigma),
		\end{equation}
		where $\chi_{\mathrm{Sym}(\boldsymbol{\alpha})}$ is the characteristic function of $\mathrm{Sym}(\boldsymbol{\alpha})$ in $\mathrm{Sym}(X)$.
		Then it follows from \eqref{eq::2} and \eqref{eq::3} that
		\begin{equation}\label{eq::4}
			\begin{aligned}
				\sum_{\sigma\in G}\Delta(\sigma)\sum_{f\in \mathcal{F}_{\sigma}}W(f)&=\sum_{\sigma\in G}\Delta(\sigma)\sum_{\boldsymbol{k}\in \mathcal{P}_m([n])}\left(\sum_{\boldsymbol{\alpha}\in \boldsymbol{k}(X)}\chi_{\mathrm{Sym}(\boldsymbol{\alpha})}(\sigma)\right)\boldsymbol{w}^{(\boldsymbol{k})}\\
				&=\sum_{\boldsymbol{k}\in \mathcal{P}_m([n])}\left(\sum_{\boldsymbol{\alpha}\in \boldsymbol{k}(X)}\sum_{\sigma\in G}\Delta(\sigma)\chi_{\mathrm{Sym}(\boldsymbol{\alpha})}(\sigma)\right)\boldsymbol{w}^{(\boldsymbol{k})}\\
				&=\sum_{\boldsymbol{k}\in \mathcal{P}_m([n])}\left(\sum_{\boldsymbol{\alpha}\in \boldsymbol{k}(X)}\sum_{\sigma\in \mathrm{Sym}(\boldsymbol{\alpha})\cap G}\Delta(\sigma)\right)\boldsymbol{w}^{(\boldsymbol{k})}.
			\end{aligned}
		\end{equation} 
		By substituting the equality from Lemma \ref{lem::key} into the left-hand side of \eqref{eq::4}, we obtain the result immediately.
	\end{proof}

As an application of Theorem  \ref{thm::main3} (or equivalently, Theorem \ref{thm::main1}), we are now ready to present the proof of Theorem \ref{thm::main2}.
	\renewcommand\proofname{\it{Proof of Theorem \ref{thm::main2}}} 
	\begin{proof}
		Let $X=[n]$, $G=\mathrm{Sym}(X)=\mathrm{Sym}(n)$, and $\Delta(\sigma)=\mathrm{sgn}(\sigma)$ for all $\sigma\in \mathrm{Sym}(n)$. By  Theorem \ref{thm::main3}, we have
		\begin{equation}\label{eq::5}
			\begin{aligned}
				\sum_{\sigma\in \mathrm{Sym}(n)}\mathrm{sgn}(\sigma)Z(\sigma,\tilde{\boldsymbol{w}})&=\sum_{\boldsymbol{k}\in \mathcal{P}_m([n])}\left(\sum_{\boldsymbol{\alpha}\in \boldsymbol{k}([n])}\sum_{\sigma\in \mathrm{Sym}(\boldsymbol{\alpha})}\mathrm{sgn}(\sigma)\right)\boldsymbol{w}^{(\boldsymbol{k})}.\\
			\end{aligned}
		\end{equation}
		For any  $\boldsymbol{k}=(k_1,\ldots,k_m)\in \mathcal{P}_m([n])$, we write $\boldsymbol{k}\preceq \boldsymbol{1}$ if $k_{j}\leq 1$ for all $j\in [m]$, and  $\boldsymbol{k}\npreceq \boldsymbol{1}$ otherwise. 
		
		Now, assume  $\boldsymbol{k}\npreceq \boldsymbol{1}$. For any  $\boldsymbol{\alpha}=(A_1,\ldots,A_m)\in \boldsymbol{k}([n])$, there exists some $i\in [m]$ such that $|A_i|\geq 2$. We can then write \[\mathrm{Sym}(\boldsymbol{\alpha})=\mathrm{Sym}(A_i)\times \prod_{j\neq i}\mathrm{Sym}(A_j).\] Thus,  
		\begin{equation*}
			\begin{aligned}
				\sum_{\sigma\in \mathrm{Sym}(\boldsymbol{\alpha})}\mathrm{sgn}(\sigma)&=\sum_{\tau\in \mathrm{Sym}(A_i),\gamma\in \prod_{j\neq i}\mathrm{Sym}(A_j)}\mathrm{sgn(\tau)}\mathrm{sgn}(\gamma)\\
				&=\left(\sum_{\tau\in \mathrm{Sym}(A_i)}\mathrm{sgn}(\tau)\right)\cdot \left(\sum_{\gamma\in \prod_{j\neq i}\mathrm{Sym}(A_j)}\mathrm{sgn}(\gamma)\right)\\
				&=0,\\
			\end{aligned}
		\end{equation*}
		where the last equality follows from the fact that  $\sum_{\tau\in \mathrm{Sym}(A_i)}\mathrm{sgn}(\tau)=0$ whenever  $|A_i|\geq 2$. Consequently, for $\boldsymbol{k}\npreceq \boldsymbol{1}$, we have  
		\[ \sum_{\boldsymbol{\alpha}\in \boldsymbol{k}([n])}\sum_{\sigma\in \mathrm{Sym}(\boldsymbol{\alpha})}\mathrm{sgn}(\sigma)=0.
		\] 
		Therefore, it suffices to  consider the case  $\boldsymbol{k}\preceq \boldsymbol{1}$ in the right-hand side of \eqref{eq::5}. In this case, $\mathrm{Sym}(\boldsymbol{\alpha})$ is  trivial since $|A_i|\leq 1$ for all $i\in [m]$. Thus, we obtain
		\begin{equation*}
			\begin{aligned}
				\sum_{\sigma\in \mathrm{Sym}(n)}\mathrm{sgn}(\sigma)Z(\sigma,\tilde{\boldsymbol{w}})
				&=\sum_{\boldsymbol{k}\in \mathcal{P}_m([n]), \boldsymbol{k}\preceq \boldsymbol{1}}\left(\sum_{\boldsymbol{\alpha}\in \boldsymbol{k}([n])}\sum_{\sigma\in \mathrm{Sym}(\boldsymbol{\alpha})}\mathrm{sgn}(\sigma)\right)\boldsymbol{w}^{(\boldsymbol{k})}\\
				&=\sum_{\boldsymbol{k}\in \mathcal{P}_m([n]),\boldsymbol{k}\preceq \boldsymbol{1}}\left(\sum_{\boldsymbol{\alpha}\in \boldsymbol{k}([n])}\sum_{\sigma=\mathrm{id}}\mathrm{sgn}(\sigma)\right)\boldsymbol{w}^{(\boldsymbol{k})}\\
				&=\sum_{\boldsymbol{k}\in \mathcal{P}_m([n]),\boldsymbol{k}\preceq \boldsymbol{1}}\left(\sum_{\boldsymbol{\alpha}\in \boldsymbol{k}([n])}1\right)\boldsymbol{w}^{(\boldsymbol{k})}\\
				&=\sum_{\boldsymbol{k}\in \mathcal{P}_m([n]),\boldsymbol{k}\preceq \boldsymbol{1}}n!\boldsymbol{w}^{(\boldsymbol{k})}\\
				&=n!\cdot e_n(\boldsymbol{w}).\\
			\end{aligned}
		\end{equation*}
		
		This completes the proof.
	\end{proof}
	
	Using Theorem \ref{thm::main2}, we can immediately deduce Corollary \ref{cor::main}.
	\renewcommand\proofname{\it{Proof of Corollary \ref{cor::main}}} 
	\begin{proof}
		Let $w_1,w_2,\ldots,w_n$ be  all the eigenvalues of $L$, and let  $ \boldsymbol{w}=(w_1,w_2,$ $\ldots,w_n)$.
		Then $e_n(\boldsymbol{w})=w_1w_2\cdots w_n$ and $\tilde{\boldsymbol{w}}=(\sum_{i=1}^n w_i,\sum_{i=1}^n w_i^2,\ldots,\sum_{i=1}^n w_i^n)=(t_1,t_2,$ $\ldots,t_n)$. By Theorem \ref{thm::main2}, we obtain
		\begin{equation*}
			w_1w_2\cdots w_n=\frac{1}{n!}\sum_{\sigma\in \mathrm{Sym}(n)}\mathrm{sgn}(\sigma)Z(\sigma,\boldsymbol{t}).
		\end{equation*}
		Since $\det(L)=w_1w_2\cdots w_n$, the result follows immediately.
	\end{proof}

	\section*{Acknowledgements}
	
	The authors are grateful to Professor T. Amdeberhan for his insightful comments and to the anonymous referees for their constructive suggestions. X. Huang was supported by National Natural Science Foundation of China (Grant No. 12471324) and Natural Science Foundation of Shanghai (Grant No. 24ZR1415500).

	\section*{Declaration of Interest Statement}
	
	The authors declare that they have no known competing financial interests or personal relationships that could have appeared to influence the work reported in this paper.

	


\begin{thebibliography}{99}\setlength{\itemsep}{0pt}
		
		
		
		\bibitem{AMD22} T. Amdeberhan, Theorems, problems and conjectures, 2012, \url{https://arxiv.org/abs/1207.4045v2}.
		
		
		

		
		\bibitem{Pol37} G. P\'{o}lya, Kombinatorische Anzahlbestimmungen f\"{u}r Gruppen, Graphen und chemische Verbindungen, Acta Math. 68 (1) (1937) 145--254.

		
		\bibitem{PR87} G. P\'{o}lya, R. C. Ronald, Combinatorial Enumeration of Groups, Graphs, and Chemical Compounds,  Springer-Verlag, New York, 1987.
		
			\bibitem{Red27} J. H. Redfield, The Theory of Group-Reduced Distributions, Amer. J. Math. 49 (3) (1927) 433--455. 
			
		\bibitem{St99} R. Stanley, Enumerative Combinatorics, vol. 2. Cambridge University Press, New York/Cambridge, 1999.
		
	\end{thebibliography}
\end{document}